\documentclass[11pt,reqno]{amsart}

\usepackage{amssymb}
\usepackage{amsfonts}
\usepackage{color}
\usepackage{enumitem}

\numberwithin{equation}{section}
\newtheorem{theorem}{Theorem}[section]
\newtheorem{corollary}[theorem]{Corollary}

\newtheorem{proposition}[theorem]{Proposition}

\theoremstyle{definition}
\newtheorem{definition}[theorem]{Definition}
\theoremstyle{remark}

\newtheorem{remark}[theorem]{Remark}

\newtheorem{example}[theorem]{Example}

\newdimen\AAdi%
\newbox\AAbo%
\def\AAk#1#2{\setbox\AAbo=\hbox{#2}\AAdi=\wd\AAbo\kern#1\AAdi{}}%

\makeatletter
\def\eqlabel#1{\def\@currentlabel{#1}}

\def\formula#1{\def\@tempa{#1}\let\@tempb\theequation\def\theequation{%
\hbox{#1}}\def\@currentlabel{(\theequation)}$$}
\def\endformula{\leqno\hbox{(\@tempa)}$$\@ignoretrue\let\theequation\@tempb}

\def\given{\hskip5\p@\relax\vrule\@width.4\p@\hskip5\p@\relax}

\newcommand{\open}[1]{%
\par\normalfont\topsep6\p@\@plus6\p@\trivlist\item[\hskip\labelsep\itshape#1%
\@addpunct{.}]\ignorespaces}

\DeclareRobustCommand{\close}[1]{%
  \ifmmode 
  \else \leavevmode\unskip\penalty9999 \hbox{}\nobreak\hfill
  \fi
  \quad\hbox{$#1$}}
\makeatother

\newlength{\toskip}

\def\<{\langle}
\def\>{\rangle}

\def \Var {\textrm{Var}}
\def \Ent {\textrm{Ent}}
\def \Osc {\textrm{Osc}}

\begin{document}
\date{\today}

\title[Inequalities for perturbed measures]{Functional inequalities for perturbed measures with applications to log-concave measures and to some Bayesian problems.}

 \author[P. Cattiaux]{\textbf{\quad {Patrick} Cattiaux $^{\spadesuit}$ \, \, }}
\address{{\bf {Patrick} CATTIAUX},\\ Institut de Math\'ematiques de Toulouse. CNRS UMR 5219. \\
Universit\'e Paul Sabatier,
\\ 118 route
de Narbonne, F-31062 Toulouse cedex 09.} \email{patrick.cattiaux@math.univ-toulouse.fr}

\author[A. Guillin]{\textbf{\quad {Arnaud} Guillin $^{\diamondsuit}$}}
\address{{\bf {Arnaud} GUILLIN},\\Universit\'e Clermont Auvergne, CNRS, LMBP, F-63000 CLERMONT-FERRAND, FRANCE.} \email{arnaud.guillin@uca.fr}

\maketitle

 \begin{center}

 \textsc{$^{\spadesuit}$  Universit\'e de Toulouse}
\smallskip


\textsc{$^{\diamondsuit}$ Universit\'e Clermont-Auvergne}
\smallskip

\end{center}

\begin{abstract}
 We study functional inequalities (Poincar\'e, Cheeger, log-Sobolev) for probability measures obtained as perturbations. Several explicit results for general measures as well as log-concave distributions are given.
The initial goal of this work was to obtain explicit bounds on the constants in view of statistical applications for instance.  These results are then applied to the Langevin Monte-Carlo method used in statistics in order to compute Bayesian estimators.
\end{abstract}
\bigskip

\textit{ Key words :}  logconcave measure, Poincar\'e inequality, Cheeger inequality, logarithmic Sobolev inequality, perturbation, bayesian statistic, sparse learning

\bigskip

\textit{ MSC 2010 : } 26D10, 39B62, 47D07, 60G10, 60J60.
\bigskip

\section{Introduction.}\label{secintro}

Let $\mu(dx)= Z_V^{-1} \, e^{-V(x)} \, dx$ be a probability measure defined on $\mathbb R^n$. We a priori do not require regularity for $V$ and allow it to take values in $\mathbb R \cup \, \{-\infty, +\infty\}$. We denote by $\mu(f)$ the integral of $f$ w.r.t. $\mu$.
\medskip

We define the Poincar\'e constant $C_P(\mu)$ as the best constant $C$ satisfying
\begin{equation}\label{eqpoinc}
\Var_\mu(f):=\mu(f^2)-\mu^2(f) \, \leq \, C \, \mu(|\nabla f|^2) \, ,
\end{equation}
for all smooth $f$, for instance all bounded $f$ with bounded derivatives. Similarly the logarithmic-Sobolev constant $C_{LS}(\mu)$ is defined as the best constant such that, for all smooth $f$ as before,
\begin{equation}\label{eqLS}
\Ent_\mu(f^2):=\mu(f^2 \, \ln(f^2))-\mu(f^2) \, \ln(\mu(f^2)) \, \leq \, C \, \mu(|\nabla f|^2) \, ,
\end{equation}
As it is well known, the Poincar\'e and the log-Sobolev constants are linked to the exponential stabilization of some markovian dynamics, like the Langevin diffusion i.e. the diffusion semi-group with generator $$A=\Delta - \nabla V.\nabla.$$ We shall give more explanations later. For simplicity we will say that $\mu$ satisfies a Poincar\'e or a log-Sobolev inequality provided $C_P(\mu)$ or $C_{LS}(\mu)$ are finite.
\medskip

One can also introduce the $\mathbb L^1$ Poincar\'e constant $C_C(\mu)$ of $\mu$ is the best constant such that, for all smooth $f$, $$\mu(|f-\mu(f)|) \, \leq \, C_C(\mu) \, \mu(|\nabla f|) \, .$$ Replacing $\mu(f)$ by $m_\mu(f)$ any $\mu$ median of $f$ defines another constant, the Cheeger constant $C'_C(\mu)$ which satisfies $\frac 12 \, C_C \leq C_C' \leq C_C$.\\

It is well known \cite{ledlogconc} that $\mathbb L^1$ and $\mathbb L^2$ Poincar\'e constants are related by the following
\begin{equation}\label{eqcheegpoinc}
C_P(\mu) \leq 4 \, (C'_C)^2(\mu) \leq 4 \, C_C^2(\mu) \, .
\end{equation}
The Cheeger constant is connected to the isoperimetric profile of $\mu$, see Ledoux \cite{ledgap}. 
\medskip

The initial goal of this work is to study the transference of these inequalities to perturbed measures. Namely, let
\begin{equation}\label{eqmuf}
\mu_F= Z_F^{-1} \, e^{-F} \, \mu
\end{equation}
be a new probability measure. The question is: what can be said for the Poincar\'e or the log-Sobolev constant of $\mu_F$ in terms of the one of $\mu$ and the properties of $F$? The question includes explicit controls, not only the finiteness of the related constants.
\medskip

This question is of course not new and some results have been obtained for a long time. We shall only recall results with explicit controls on the constants. The most famous is certainly the following general result of Holley and Stroock (see for example \cite{BaGLbook})
\begin{theorem}\label{thmholley}
If $F$ is bounded, then
$$C_P(\mu_F) \leq e^{\Osc F} \, C_P(\mu), \quad \textrm{ and } \quad 
C_{LS}(\mu_F) \leq e^{\Osc F} \, C_{LS}(\mu).$$
\end{theorem}
Other results, where the constants are not easy to trace, have been obtained in \cite{AS} (also see \cite{CatPota} section 7). The result reads as follows: if $\mu$ satisfies a log-Sobolev inequality, and $e^{|\nabla F|^2}$ belongs to all the $\mathbb L^p(\mu)$ for $p<+\infty$, then $\mu_F$ also satisfies a log-Sobolev inequality. In particular the result holds true if $F$ is Lipschitz. We shall revisit this result in subsection \ref{subsecgenels}. For the Poincar\'e inequality, some general results have been obtained in \cite{gongwu} (see e.g. \cite{CatPota} proposition 4.4 for a simplified formulation). 
\medskip

Most of the other known results assume some convexity property.

We shall say that $\mu$ is \emph{log-concave} if $V$ is a convex function defined on some convex subset $U$. Since $V$ can be infinite, this definition contains in particular the uniform measure on a convex body.

If $V$ is strongly convex, i.e. for all $x \in \mathbb R^n$, $\langle u \, , \, Hess_V(x) \, u\rangle \geq \rho \, |u|^2$, where $\langle .,.\rangle$ denotes the euclidean scalar product and $Hess_V(x)$ the Hessian of $V$ computed at point $x$, a consequence of Brascamp-Lieb inequality (\cite{BrasLieb}) is the inequality
\begin{equation}\label{eqBras}
C_P(\mu) \, \leq \, 1/\rho \, .
\end{equation}
This relation was extended to more general situations and is often called the Bakry-Emery criterion or the curvature-dimension criterion $CD(\rho,\infty)$ (see \cite{BaGLbook} for a complete description of curvature-dimension criteria). Some improvements for variable curvature bounds are contained in \cite{CFG}. Recall that the Bakry-Emery approach allows to show that in the strongly convex situation 
\begin{equation}\label{eqBE}
C_{LS}(\mu) \, \leq \, 2/\rho \, .
\end{equation}
Combining these results with the Holley-Stroock perturbation result shows that one can relax the strong convexity assumption in a bounded subset (i.e. assume strong convexity at infinity only).

That $C_P(\mu)<+\infty$ for general log-concave measures was first shown in 1999 by S. Bobkov in \cite{bob99}. Another proof was given in \cite{BBCG} using Lyapunov functions, as introduced in \cite{BCG}. 
\medskip

Once one knows that $C_P(\mu)$ is finite, a particularly important problem is to get some explicit estimates. A celebrated conjecture due to Kannan, Lov\'asz and Simonovits (KLS for short) is that there exists a universal constant $C$ such that
\begin{equation}\label{eqKLS}
\sigma^2(\mu) \, \leq \, C_P(\mu) \, \leq \, C \, \sigma^2(\mu)
\end{equation}
where $\sigma^2(\mu)$ denotes the largest eigenvalue of the Covariance matrix $Cov_{i,j}(\mu)=Cov_\mu(x_i,x_j)$ and $\mu$ is log-concave. The left hand side is immediate. Since, a lot of works have been devoted to this conjecture, satisfied in some special cases. We refer to the book \cite{Alonbast} for references before 2015, and to \cite{CGlogconc} for more information on the Poincar\'e constant of log-concave measures. Up to very recently the best general known result was 
\begin{equation}\label{eqKLS-LV}
C_P(\mu) \, \leq \, C \, n^{\frac 12} \, \sigma^2(\mu)
\end{equation}
as a consequence of the results by Lee and Vempala (\cite{leevempfoc}). It has been very recently announced (see \cite{chen} Theorem 1) a drastically better bound namely the existence of an universal constant $C$ such that 
\begin{equation}\label{eqKLS-chen}
C_P(\mu) \, \leq \, e^{C \sqrt{\ln(n) \, \ln(1+\ln(n))}} \, \sigma^2(\mu) \, .
\end{equation}
The dimension dependence of such results is what is important. Recall that for $n=1$ one knows that $C \leq 12$ according to Bobkov's result (see \cite{bob99} corollary 4.3).
\medskip

In this framework, complementary perturbation results have been shown
\begin{theorem}\label{thmrappels}
\begin{enumerate}
\item[(1)] \; \emph{(Miclo, see lemma 2.1 in Bardet et al \cite{BGMZ})} \quad If $Hess V \geq \rho \, Id$ for some $\rho>0$ and $F$ is $L$-Lipschitz then $$C_P(\mu_F) \, \leq \, \frac{2}{\rho} \, e^{4 \sqrt{2n/\pi} \, \frac{L^2}{\rho^2}} \, .$$
\item[(2)] \; \emph{(see \cite{CGsemin} example (3) section 7.1)} \quad With the same assumptions as in (1), $$C_P(\mu_F) \, \leq \, \frac 12 \, \left(\frac{2L}{\rho} \, + \, \sqrt{\frac {8}{\rho}}\right)^2 \, e^{L^2/2\rho} \, .$$
\item[(3)] \; \emph{(Barthe-Milman \cite{barmil} Theorem 2.8)} \quad If $\mu_F$ is log-concave and $\mu_F(e^{-F} > K \, \mu(e^{-F})) \leq \frac 18$ then $$C_P(\mu_F) \leq C^2 (1+\ln K)^2 \, C_P(\mu) \, .$$ Here $C$ is a universal constant.
\end{enumerate}
\end{theorem}
The final result (3) is the most general one obtained by transference in the log-concave situation. \cite{barmil} contains a lot of other results in this direction, \cite{CGlogconc} contains alternative, simpler but worse results. Notice that (3) is wrongly recalled in  \cite{CGlogconc} where a square is missing. The remarkable property of (3) or (2) is that the bound is dimension free (but actually dimension is hidden either in the Lipschitz constant or in the choice of K in most of the generic examples). One can find various other perturbation results in \cite{BLW07} relying on growth conditions, and usually stronger inequalities to get weaker ones.
\medskip

In \cite{CGlogconc} we have studied several properties of the Poincar\'e constant for log-concave measures in particular the transference of these inequalities using absolute continuity, distance, mollification. This study was based on the fact that weak forms of the Poincar\'e inequality imply the usual form. A similar result was first stated by E. Milman (\cite{emil1}) and the section 9.2 in \cite{CGlogconc} is devoted to extensions of E. Milman's results. Actually, in subsection 9.3.1 of \cite{CGlogconc} devoted to the transference via absolute continuity we missed the point. As explained in \cite{barmil} the right way to obtain good results is to use the concentration properties of the initial measure (see the proof of Theorem 2.7 in \cite{barmil}). The main (only) default in \cite{emil1,barmil} is that these papers obtain results up to universal constants that are difficult to trace. Thanks to the weak Poincar\'e inequalities used in \cite{CGlogconc}, it is possible to obtain explicit bounds for these universal constants, sometimes up to a small loss. Why are we so interested in numerical bounds ?
\medskip

The motivation of this note came from a statistical question. Indeed, log-concave distributions recently deserve attention in Statistics, see e.g. the survey \cite{SW}. Our starting point was a question asked to us by S. Gadat on the work \cite{DT12} by Dalalyan and Tsybakov on sparse regression learning, we shall recall in more details in section \ref{secsparse}. The question should be formulated as follows. Let 
\begin{equation}\label{eqmuf}
\mu_F= Z_F^{-1} \, e^{-F} \, \mu
\end{equation}
be a new probability measure. Assume that $\mu$ is log-concave and that $F$ is convex (log-concave perturbation of a log-concave measure). Is it possible to control $C_P(\mu_F)$ by $C_P(\mu)$ at least up to an universal multiplicative constant ? Since in a sense $\mu_F$ is ``more'' log-concave than $\mu$, such a statement seems plausible, at least when the ``arg-infimum'' of $F$ coincindes with the one of $\mu$. 

A first partial answer was obtained by F. Barthe and B. Klartag in \cite{BK19} Theorem 1: 
\begin{theorem}\label{thmbarklar}
For $n\geq 2$, choose $V(x)=\sum_{i=1}^n \, |x_i|^p$ for some $1\leq p \leq 2$, and assume that $F$ is an even convex function then $$C_P(\mu_F) \leq C \, (\ln(n))^{\frac{2-p}{p}} \, C_P(\mu) \, ,$$ where $C$ is some universal constant.
\end{theorem}
Of course here $C_P(\mu)$ does not depend on the dimension since $\mu$ is a product measure and $$C_P(\mu_1 \otimes ... \otimes \mu_k) \leq \max_{j=1,...,k} \, C_P(\mu_j) \, .$$ Unfortunately this result does not apply to sparse regression as in \cite{DT12}, where $F$ is not even.

We will thus first study perturbation for log-concave measures in section \ref{secperturb} and then see how it can be applied to the aforementioned statistical question. For the latter explicit numerical bounds are of key interest.
\bigskip

We will now describe the contents of the present paper.
\medskip

In the next section we describe general perturbation results both for the Poincar\'e and the log-Sobolev constants. The naive method we are using does not seem to have been explored with the exception of some results for log-Sobolev contained in \cite{AS,gongwu,CatToul}. The results can be summarized as follows: for a not too big Lipschitz perturbation $F$, $\mu_F$ one may explicitly compare $C_P(\mu_F)$ and $C_P(\mu)$, the same for the log-Sobolev constants. If $F$ is $C^2$ one can 
replace the Lipschitz norm by a bound for $AF - \frac 12 |\nabla F|^2$. 

In subsection \ref{subsecmollif} we give a surprisingly simple application to mollified measures, extending part of the results in \cite{BGMZ}.
\medskip

In section \ref{secperturb} we show how to improve these results when $\mu_F$ is log-concave. We first study how to get explicit controls in the results obtained by Barthe and Milman \cite{barmil} using concentration, by plugging our explicit results of \cite{CGlogconc}. We then transpose the results of the first section. For short, if in general one needs to control the uniform norm of $\nabla F$, in the log-concave case it is enough to control its $\mathbb L^2$ norm.

As an immediate consequence, we obtain in subsection \ref{subsecgauss2}, explicit controls for the pre-constants in results by Bobkov \cite{bob99}. The idea of proof is to consider $\mu(dx)=e^{-V(x)} dx$ as a perturbation $\nu_F$ of $\nu(dx) =e^{-V(x)-\frac 12 \, \rho |x|^2} \, dx$ that satisfies Bakry-Emery criterion. Actually the constants we are obtaining here are worse than the known ones, but the methodology will be used in the sequel. 

In subsection \ref{subsecsubbot} we replace the gaussian perturbation $|x|^2$ by $\sum_i |x_i|^p$ for $p >2$. Using another result obtained in \cite{BK19} for unconditional measures, we recover the control $C \ln^2(n)$ of the Poincar\'e constant of unconditional log-concave measures first obtained by Klartag (\cite{Klartuncond}) we already recovered in \cite{CGlogconc}. The advantage of our method is that it furnishes an explicit bound for the constant $C$.
\medskip

The final section is devoted to the application to two statistical models: linear regression, following \cite{DT12} and identification as in \cite{GPP}. In the linear regression case we give explicit controls for the rate of convergence of the Langevin Monte-Carlo algorithm proposed in \cite{DT12}, that are much better than the ones suggested therein.
\bigskip

\section{Smooth perturbations in general.}\label{secgene}

\subsection{Poincar\'e inequalities. \\ \\}\label{subsecgenepoinc}

We shall follow a direct perturbation approach, as the one of \cite{CatToul} section 4 used for the log-Sobolev constant. Since we shall use a similar but slightly different approach in the next sections we first isolate the starting point of the proof.

For a smooth $f$ and a constant $a$ we may write
\begin{equation*}
\Var_{\mu_F}(f) \, \leq \, \mu_F((f-a)^2) \, = \,  \mu^{-1}(e^{-F}) \, \mu\left(((f-a) e^{- \frac 12 \, F})^2\right) \, .
\end{equation*}
We choose $$a=\frac{\mu\left(f \, e^{- \frac 12 \, F}\right)}{\mu\left(e^{- \frac 12 \, F}\right)}$$ so that the function $(f-a) \, e^{- \frac 12 \, F}$ is $\mu$ centered. One can thus use the Poincar\'e inequality for $\mu$ in order to get, for all $a$ and all $\varepsilon >0$,
\begin{equation}\label{eqnaive1}
\mu_F((f-a)^2) \, \leq \, C_P(\mu) \, \int \, |\nabla f \, - \, \frac 12 \, (f-a) \, \nabla F|^2 \, d\mu_F \, .
\end{equation}
One deduces, for all $\varepsilon >0$,
\begin{equation}\label{eqnaive1bis}
\mu_F((f-a)^2) \, \leq \, C_P(\mu) \, \left((1+\varepsilon^{-1}) \, \mu_F(|\nabla f|^2) + \frac {1+\varepsilon}{4} \, \mu_F((f-a)^2 \, |\nabla F|^2)\right) \, .
\end{equation}
We thus may state
\begin{theorem}\label{thmperturbnaivegene1}
If $F$ is $L$-Lipschitz on the support of $\mu$ and if there exists $\epsilon>0$ such that
$$ s:= \frac 14(1+\epsilon) \, C_P(\mu) \, L^2<1,$$
then 
$$C_P(\mu_F)\le \frac{(1+\epsilon^{-1})C_P(\mu)}{1- \, s}.$$ 
\end{theorem}
\begin{proof}
Let $a$ be as before. If $F$ is $L$-Lipschitz we deduce from \eqref{eqnaive1bis}, $$\mu_F((f-a)^2) \, \leq \, (1+\varepsilon^{-1}) \, C_P(\mu) \, \mu_F(|\nabla f|^2) \, + \, \frac{1+\varepsilon}{4} \, C_P(\mu) \, L^2 \, \mu_F((f-a)^2) \, ,$$ so that $$\Var_{\mu_F}(f) \, \leq \, \mu_F((f-a)^2) \, \leq \, \frac{(1+\varepsilon^{-1})C_P(\mu)}{1-\, \frac 14(1+\varepsilon) \, C_P(\mu) \, L^2} \, \mu_F(|\nabla f|^2)$$
as soon as $\exists \, \varepsilon>0$ such that 
$$ \frac 14(1+\varepsilon) \, C_P(\mu) \, L^2 \, <1.$$ 
\end{proof}
This theorem is quite sharp. Indeed recall \cite[Prop. 4.4.2]{BaGLbook}: for every $1$-Lipshitz function $f$ and $s<\sqrt{4/C_P(\mu)}$, one has $\mu(e^{sf})<\infty$. So in a sense our Lipschitz perturbation is nearly the largest one preserving that $\mu_F$ is well defined. 
\smallskip

Similarly for the Cheeger constant we may state
\begin{proposition}\label{propperturbnaivecheeger}
If $F$ is $L$-Lipschitz on the support of $\mu$ and $C_C(\mu)\,L<1$, $$C'_C(\mu_F)\le \frac{C_C(\mu)}{1-C_C(\mu)\,L}.$$
\end{proposition}
\begin{proof}
First, for all $a$,  $$\mu_F(|f-m_{\mu_F}(f)|) \leq \mu_F(|f-a|) = \mu^{-1}(e^{-F}) \, \mu\left(|(f-a)e^{-F}|\right)$$ so that choosing this time $$a=\frac{\mu\left(f \, e^{- \, F}\right)}{\mu\left(e^{- \, F}\right)}$$  we get the result for the Cheeger constant by using 
\begin{eqnarray*}
 \mu_F(|f-a|) &\le& C_C(\mu) \, \left(\mu_F(|\nabla f|) \, + \, \mu_F(|f-a| \, |\nabla F|)\right)\\
 &\le&C_C(\mu) \, \mu_F(|\nabla f|) + C_C(\mu) \, L \, \mu_F(|f-a|) \, .
\end{eqnarray*}
\end{proof}

The proof of these two statement is so simple that we cannot believe that the result is not known. In fact, the only comparable result we found, using seemingly more intricate techniques, is in \cite[Th. 2.7]{AS} but with roughly a factor 16 in our favor and an explicit Poincar\'e constant.
\smallskip

What is remarkable is that we may even allow more general condition than $F$ to be $L$-Lipschitz however assuming more regularity for $F$. It requires to be careful when $\mu$ is compactly supported.
\begin{theorem}\label{thmperturbnaivegene2}
a) If $V$ is of $C^1$ class  and $F$ is of $C^2$ class and satisfies for some $\varepsilon>0$
$$ C_P(\mu) \, \sup_x \left(AF -\frac 12|\nabla F|^2\right)_+(x) \leq  2(1-\varepsilon) \, ,$$
where $u_+=\max(u,0)$, then 
$$C_P(\mu_F) \leq \frac{C_P(\mu)}{\varepsilon}.$$ 
b) Assume that $U:=\{V<+\infty\}$ is an open subset with a smooth boundary $\partial U$ and that $V$ is of $C^1$ class in $U$.  Let $F$ be of $C^2$ class and such that $\partial_n F \geq 0$ on $\partial U$ where $\partial_n$ denotes the normal derivative pointing outward. If $F$ satisfies for some $\varepsilon>0$
$$ C_P(\mu) \, \sup_{x \in U} \left(AF -\frac 12|\nabla F|^2\right)_+(x) \leq  2(1-\varepsilon) \, ,$$
where $u_+=\max(u,0)$, then 
$$C_P(\mu_F) \leq \frac{C_P(\mu)}{\varepsilon}.$$ 
\end{theorem}

\begin{proof}
a) If one allows $F$ to be more regular, one can replace \eqref{eqnaive1bis} by another inequality. Indeed starting with \eqref{eqnaive1} we have
\begin{eqnarray}\label{eqnaive3}
\mu_F((f-a)^2) &\leq& C_P(\mu) \, \int |\nabla f \, - \, \frac 12 (f-a) \nabla F|^2 \, d\mu_F \nonumber \\ &\leq& C_P(\mu) \, \left(\mu_F(|\nabla f|^2) \, - \, \frac 12 \, \mu_F(\nabla(f-a)^2 \, . \, \nabla F) \, + \, \frac 14 \, \mu_F((f-a)^2 \, |\nabla F|^2)\right) \nonumber \\ &\leq& C_P(\mu) \, \mu_F(|\nabla f|^2) \, + \, \frac 12 \, C_P(\mu) \, \mu_F\left((f-a)^2 \, \left[A_F \, F \, + \, \frac 12 \, |\nabla F|^2 \right]\right)
\end{eqnarray}
where $A_F=A - \nabla F.\nabla=\Delta - \nabla V. \nabla - \nabla F.\nabla$. Finally
\begin{equation}\label{eqnaive3bis}
\mu_F((f-a)^2) \, \leq \, C_P(\mu) \, \mu_F(|\nabla f|^2) \, + \, \frac 12 \, C_P(\mu) \, \mu_F\left((f-a)^2 \, \left[A \, F \, - \, \frac 12 \, |\nabla F|^2 \right]\right) \, .
\end{equation}
b) We have to start again with \eqref{eqnaive1} where integration holds in $U$. This yields
\begin{eqnarray}\label{eqnaivekls}
\mu_F(f-a)^2) &\leq& C_P(\mu) \, \int_U |\nabla f \, - \, \frac 12 (f-a) \nabla F|^2 \, d\mu_F \\ &\leq& C_P(\mu) \, \left(\mu_F(|\nabla f|^2) \, - \, \frac 12 \, \mu_F(\langle \nabla(f-a)^2 \, , \, \nabla F\rangle) \, + \, \frac 14 \, \mu_F((f-a)^2 \, |\nabla F|^2)\right) \, .\nonumber
\end{eqnarray}
To control the second term we have to use Green's formula to integrate by parts
\begin{equation}\label{eqgreen}
\mu_F(\langle \nabla(f-a)^2 \, , \, \nabla F\rangle) = - \, \mu_F((f-a)^2 \, A F) \, + \, \mu_F^\partial((f-a)^2 \, \partial_n \, F)
\end{equation}
where $A=\Delta - \nabla V.\nabla$, $\mu_F^\partial$ denotes the surface measure on $\partial U$ and $\partial_n$ denotes the normal derivative pointing outward. The end of the proof is then similar.
\end{proof}
\begin{example}
Let us describe a very simple case which illustrates the difference between Th.\ref{thmperturbnaivegene1} and Th.\ref{thmperturbnaivegene2}. Let $\mu=\frac12 e^{-|x|}d\mu$ for which $C_P(\mu)=4$, and consider $F(x)=\rho |x|$ which is $|\rho|$-Lipschitz. An application of Th.\ref{thmperturbnaivegene1}  shows that $\mu_F$ still satisfies a Poincar\'e inequality if $|\rho|<1$ thereas Th.\ref{thmperturbnaivegene1} implies that $\mu_F$ satsfies a Poincar\'e inequality as soon as $\rho>-1$ which is optimal in this case. \hfill $\diamondsuit$
\end{example}

\begin{example}\label{exgaussgene}
For $\rho \in \mathbb R^+$ consider $$\mu^\rho(dx) \, = \, Z_\rho^{-1} \, e^{-V(x) - \, \frac{\rho}{2} |x|^2} \, dx \, ,$$ i.e. $$F(x) = \frac{\rho}{2} \, |x|^2 \, $$
which is not Lipschitzian. We thus have $$(AF - \frac 12 \, |\nabla F|^2)(x) = \rho \, \left(n \, - \, x.\nabla V(x) \, - \, \frac 12 \, \rho \, |x|^2\right) \, .$$ Hence if $x.\nabla V \geq - K-K'|x|^2$, $AF - \frac 12 \, |\nabla F|^2 \leq \rho(n+K)+(K'-\rho/2)|x|^2$ so that $C_P(\mu^\rho) \leq \frac{C_P(\mu)}{1-\varepsilon}$ as soon as $\rho \leq \frac{2(1-\varepsilon)}{C_P(\mu) \, (n+K)}$ and $\rho>K'$. 

Thus a (very) small gaussian perturbation of a measure satisfying some Poincar\'e inequality is still satisfying some Poincar\'e inequality. Though natural, we do not know any other way to prove such a result. \hfill $\diamondsuit$
\end{example}

\medskip

This result can be compared with the one in \cite{gongwu} where $\mu$ is assumed to satisfy a log-Sobolev inequality. Actually, one can almost recover Gong-Wu result. Indeed, in \eqref{eqnaive1bis} and \eqref{eqnaive3bis}, the final step requires to control a term in the form $$\mu_F((f-a)^2 \, G) \, .$$ According to the variational definition of relative entropy we have, for $\alpha>0$,
\begin{equation}\label{eqdefentrop}
\mu_F((f-a)^2 \, G) \leq \frac{1}{\alpha} \, \Ent_{\mu_F}((f-a)^2) \, + \, \frac{1}{\alpha} \, \mu_F((f-a)^2) \, \ln \mu_F(e^{\alpha G}) \, .
\end{equation}
Replacing $F$ by $F + \ln(\mu(e^{-F}))$ we may assume for simplicity that $\mu(e^{-F})=1$. Defining $g=e^{- \, F/2} \, (f-a)$ we have $\mu(g^2)=\mu_F((f-a)^2)$. Hence for all $\theta >0$, 
\begin{eqnarray}\label{eqls}
\Ent_{\mu_F}((f-a)^2) \, &=& \, \Ent_\mu(g^2) \, + \, \mu(g^2 \, F) \nonumber \\ &\leq& C_{LS}(\mu) \, \mu(|\nabla g|^2) \, + \, \mu(g^2 \, F)  \\ &\leq& C_{LS}(\mu) \, (1+\theta^{-1}) \, \mu_F(|\nabla f|^2) \, + \, C_{LS}(\mu) \, \frac{1+\theta}{4} \, \mu_F((f-a)^2 \, |\nabla F|^2) \,  \nonumber \\ & & \quad + \, \mu_F((f-a)^2 \, F) \, , \nonumber
\end{eqnarray}
if we follow the proof of Theorem \ref{thmperturbnaivegene1}, or
\begin{equation}\label{eqls2}
\Ent_{\mu_F}((f-a)^2) \leq C_{LS}(\mu) \, \mu_F(|\nabla f|^2) \, + \, C_{LS}(\mu) \, \mu_F\left((f-a)^2 \, [AF-\frac 12 \, |\nabla F|^2]\right) \, + \, \mu_F((f-a)^2 \, F)
\end{equation}
if we follow the proof of Theorem \ref{thmperturbnaivegene2}. Using again \eqref{eqdefentrop} we have obtained
\begin{enumerate}
\item[(1)] \; if for some positive $s$ and $t$, 
$$\frac 1s + \frac{1+\theta}{4t} \, C_{LS}(\mu):=D_1 \leq 1
$$
then $$\Ent_{\mu_F}((f-a)^2) \leq \frac{1}{1-D_1} \, \frac{C_{LS}(\mu)}{1+\theta} \, \mu_F(|\nabla f|^2) + $$ $$\quad + \, \frac{1}{1-D_1} \, \mu_F((f-a)^2) \left[\frac 1s \, \ln \mu_F(e^{sF}) + \frac{(1+\theta)C_{LS}(\mu)}{4t} \, \ln \mu_F(e^{t|\nabla F|^2})\right].$$ 
\item[(2)] \; if for some positive $s$ and $t$,
$$\frac 1s + \frac{1}{t} \, C_{LS}(\mu):=D_2 \leq 1
$$
then $$\Ent_{\mu_F}((f-a)^2) \leq \frac{1}{1-D_2} \, C_{LS}(\mu) \, \mu_F(|\nabla f|^2) + $$ $$\quad + \, \frac{1}{1-D_2} \, \mu_F((f-a)^2) \left[\frac 1s \, \ln \mu_F(e^{sF}) + \frac{C_{LS}(\mu)}{t} \, \ln \mu_F\left(e^{t[AF-\frac 12 |\nabla F|^2]}\right)\right].$$ 
\end{enumerate}
Finally we have
\begin{proposition}\label{propgenemoche1}
We suupose here that $\mu$ satisfies a logarithmic Sobolev inequality and thus that $C_{LS}(\mu)$is finite.\\
(i) Assume that for some positive $s$ and $t$, 
\begin{equation}\label{eqdefD1}
\frac 1s + \frac{1+\theta}{4t} \, C_{LS}(\mu):=D_1 \leq 1
\end{equation} Assume in addition that there exist $\alpha>0$, $\varepsilon>0$ and $\theta>0$ such that $$T'_1:=\frac{(1+\varepsilon)C_P(\mu)}{4 \alpha} \, T_1 < 1$$ where  $$T_1:= \ln \mu_F(e^{\alpha |\nabla F|^2}) + \frac{1}{1-D_1}  \left[\frac 1s  \ln \mu_F(e^{sF})+ \ln(\mu(e^{-F})) + \frac{(1+\theta)C_{LS}(\mu)}{4t} \ln \mu_F(e^{t|\nabla F|^2})\right].$$ Then, $$C_P(\mu_F) \, \leq \, \frac{1}{1-T'_1} \, C_P(\mu) \, \left((1+\varepsilon^{-1}) + \frac{(1+\theta^{-1})(1+\varepsilon)}{4 \alpha} \, C_{LS}(\mu)\right) \, .$$
(ii) Assume \for some positive $s$ and $t$,
\begin{equation}\label{eqdefD2}
\frac 1s + \frac{1}{t} \, C_{LS}(\mu):=D_2 \leq 1
\end{equation}
 Assume in addition that there exists $\alpha>0$  such that $$T'_2:= \frac{C_P(\mu)}{2\alpha} \, T_2 \, < \, 1$$ where $$T_2:=\ln \mu_F(e^{\alpha[AF-\frac 12 |\nabla F|^2]}) + \, \frac{1}{1-D_2} \, \left[\frac 1s \, \ln \mu_F(e^{sF})+ \ln(\mu(e^{-F}))  + \frac{C_{LS}(\mu)}{t} \, \ln \mu_F\left(e^{t[AF-\frac 12 |\nabla F|^2]}\right)\right].$$ Then $$C_P(\mu_F) \, \leq \, \frac{1}{1-T'_2} \, C_P(\mu)\, \left(1 + \frac{C_{LS}(\mu)}{\alpha}\right) \, . $$
\end{proposition}
Of course such a result is difficult to apply, but the method will be useful to get a perturbation result for the log-Sobolev constant. Part of the result has been described in \cite{CatToul}. Note however that as $\mu$ satisfies a logarithmic Sobolev inequality then one has Gaussian integrability properties, so that at least for every $s<1/C_{LS}(\mu)$
$$\int e^{s|x|^2}d\mu<\infty$$
and thus if for some positive $a,b$ sufficiently small, one has $|F|,|\nabla F|^2<a+b|x|^2$, $T_1$ is then finite and can be made explicit.
\medskip

\subsection{Log-Sobolev inequality. \\ \\}\label{subsecgenels}

We can now similarly look at the log-Sobolev constant

\begin{theorem}\label{thmnaivelogsob}
Assume that $F$ is $L$-Lipschitz on the support of $\mu$ and that $\mu$ satisfies a log-Sobolev inequality with constant $C_{LS}(\mu)$. Also assume for simplicity that $\mu(e^{-F})=1$. 
\begin{enumerate}
\item[(1)] \quad If $\sup_{x \in supp(\mu)} \, F(x) = M$, then for all $\theta >0$, $$C_{LS}(\mu_F) \, \leq \, (1+ \theta^{-1}) \, C_{LS}(\mu) \, + \, C_P(\mu_F) \, \left(\frac{1+\theta}{4} \, L^2 \, C_{LS}(\mu) \, + \, M \, + \, 2\right) \, .$$ 
\item[(2)] \quad For all $\beta >0$, for all $\theta>0$, 
\begin{eqnarray*}
C_{LS}(\mu_F) \, &\leq& \,  \frac{(\beta +1)(1+ \theta^{-1})}{\beta} \, C_{LS}(\mu) \,  + \, C_P(\mu_F) \, (2 + \mu(F))  \\ & & \quad + \, L^2 \, C_P(\mu_F) \, C_{LS}(\mu) \left(\frac{(1+\theta)(1+\beta)}{4\beta} + \frac{\beta^2}{2}\right) \, .
\end{eqnarray*}
\end{enumerate}
If in addition the condition in Theorem \ref{thmperturbnaivegene1} is satisfied for some $\varepsilon >0$ we may replace $C_P(\mu_F)$ by the bound obtained in Theorem \ref{thmperturbnaivegene1}.
\end{theorem}
Remark that the first stement does not enter the framework of Holley-Stroock's theorem as only a one sided bound is assumed on $F$.
\begin{proof}
Let $f$ be smooth and such that $\mu_F(f^2)=1$. Defining $g=e^{- \, F/2} \, f$ we have $\mu(g^2)=1$. It follows for all $\theta >0$, 
\begin{eqnarray}\label{eqls}
\Ent_{\mu_F}(f^2) \, &=& \, \mu_F(f^2 \, \ln(f^2)) \, = \, \mu(g^2 \, \ln(g^2)) \, + \, \mu(g^2 \, F) \nonumber \\ &\leq& C_{LS}(\mu) \, \mu(|\nabla g|^2) \, + \, \mu(g^2 \, F)  \\ &\leq& C_{LS}(\mu) \, (1+\theta^{-1}) \, \mu_F(|\nabla f|^2) \, + \, C_{LS}(\mu) \, \frac{1+\theta}{4} \, \mu_F(f^2 \, |\nabla F|^2) \, + \, \mu_F(f^2 \, F) \, . \nonumber
\end{eqnarray}
In the first case, we can bound the sum of the last two terms by $$\left(C_{LS}(\mu) \, \frac{1+\theta}{4} \, L^2 + M\right) \, \mu_F(f^2)$$ and apply the Poincar\'e inequality for $\mu_F$ provided $\mu_F(f)=0$. To conclude it is then enough to recall Rothaus lemma (Lemma 5.14 in \cite{BaGLbook}), $$\Ent_{\nu}(f^2) \, \leq \, \Ent_\nu((f-\nu(f))^2) \, + \, 2 \, \Var_\nu(f) \, .$$
If $F$ is not bounded above we can use the variational definition of relative entropy as before: 
 $$\mu_F(f^2 \, F) \, \leq \, \frac 1\alpha \, \Ent_{\mu_F}(f^2) \, + \, \frac 1\alpha \, \mu_F(f^2) \, \ln (\mu_F(e^{\alpha F})) \, .$$ 
 Gathering all the previous bounds we have obtained, provided $\alpha>1$,
\begin{equation}\label{eqboundls}
\Ent_{\mu_F}(f^2) \, \leq \, \frac{\alpha}{\alpha -1} \, \left(C_{LS}(\mu) \, (1+\theta^{-1}) \, \mu_F(|\nabla f|^2) \, + \, C \, \mu_F(f^2)\right) 
\end{equation}
with 
$$C \, =  \, C_{LS}(\mu) \, \frac{1+\theta}{4} \, L^2 \, + \, \frac 1\alpha \, \ln(\mu_F(e^{\alpha F})) \, . $$
We may then argue as before using Rothaus lemma again. The bound $$\mu_F(e^{\alpha F})= \mu(e^{(\alpha-1) F}) \leq e^{(\alpha-1)  \mu(F) + (C_{LS}(\mu) \, L^2 \, (\alpha-1)^2 /2)}$$ is known as the Herbst argument (see e.g. \cite{BaGLbook} Proposition 5.4.1). 
\medskip
\end{proof}

As before we may replace the Lipschitz assumption by an integrability condition yielding the next result whose proof, similar to the previous one, is omitted 

\begin{theorem}\label{thmlsgeneaida}
Suppose that $\mu$ satisfies a logarithmic Sobolev inequality with constant $C_{LS}(\mu)$ and that $\mu(e^{-F})=1$. Assume that there exist $\alpha>1$ and $\beta,\theta>0$ such that
$$\mu_F(e^{\alpha F})<\infty,\qquad \mu_F(e^{\beta|\nabla F|^2})<\infty$$
and
$$C_{LS}(\mu)\,\frac{1+\theta}{4\beta}+\frac1\alpha := \delta <1$$
then $\mu_F$ also satisfies a logarithmic Sobolev inequality with constant $C_{LS}(\mu_F)$ equal to 
$$\frac{1}{1-\delta}\left[C_{LS}(\mu)(1+\theta^{-1})+C_P(\mu_F)\left(2 +C_{LS}(\mu)\frac{1+\theta}{4\beta}\log \mu_F(e^{\beta|\nabla F|^2})+\frac1\alpha\log\mu_F(e^{\alpha F})\right)\right].$$
\end{theorem}
The previous Theorem is a version of the one obtained in \cite{AS} as recalled in the introduction.
\medskip

Finally, if $F$ is more regular we may replace \eqref{eqls} by the following $$\Ent_{\mu_F}(f^2) \leq C_{LS}(\mu) \mu(|\nabla f|^2) + \frac 12 \, C_{LS}(\mu) \, \mu_F\left(f^2 \, [AF-\frac 12 |\nabla F|^2]\right) + \mu_F(f^2 \, F) \, .$$ Arguing as before we thus obtain
\begin{theorem}\label{thmlsgenetoul}
Suppose that $\mu$ satisfies a logarithmic Sobolev inequality with constant $C_{LS}(\mu)$ and that $V$ is $C^1$. Assume that there exist $\alpha>1$ and $\beta>0$ such that
$$\mu_F(e^{\alpha F})<\infty,\qquad \mu_F(e^{\beta [AF - \frac 12 |\nabla F|^2]})<\infty$$
and
$$C_{LS}(\mu)\,\frac{1}{2\beta}+\frac1\alpha := \delta <1$$
then $\mu_F$ also satisfies a logarithmic Sobolev inequality with constant $C_{LS}(\mu_F)$ equal to 
$$\frac{1}{1-\delta}\left[C_{LS}(\mu)+C_P(\mu_F)\left(2 +C_{LS}(\mu)\frac{1}{2\beta}\log \mu_F(e^{\beta [AF - \frac 12 |\nabla F|^2]})+\frac1\alpha\log\mu_F(e^{\alpha F})\right)\right].$$
If in addition the condition in Theorem \ref{thmperturbnaivegene2} is satisfied for some $\varepsilon >0$ we may replace $C_P(\mu_F)$ by the bound obtained in Theorem \ref{thmperturbnaivegene2}.
\end{theorem}
\medskip

\begin{remark}
Notice that if $AF - \frac 12 |\nabla F|^2$ is non-positive at infinity (which is often the case in concrete examples), and $F$ is $C^2$, $e^{\beta [AF - \frac 12 |\nabla F|^2]}$ is bounded for all $\beta>0$, so that the condition in the previous theorem reduces to the integrability of $e^{\alpha F}$ for some $\alpha >1$. The most stringent condition is thus the one in theorem \ref{thmperturbnaivegene2} ensuring the finiteness of the Poincar\'e constant. \hfill $\diamondsuit$
\end{remark}
\medskip

\subsection{Application to mollified measures. \\ \\}\label{subsecmollif}

Let $\nu$ be a given probability measure (non necessarily absolutely continuous) and define $\nu^\sigma$ as the convolution $\nu^\sigma = \nu * \gamma_\sigma$ where $\gamma_\sigma$ denotes the centered gaussian distribution with covariance matrix $\sigma^2 \, Id$. In other words $\nu^\sigma$ is the probability distribution of $X + \sigma G$ where $X$ is a random variable with distribution $\nu$ and $G$ is a standard gaussian variable. A natural question is to know when $\nu^\sigma$ satisfies a Poincar\'e or a log-Sobolev inequality and to get some controls on the corresponding constants. Notice that $\nu$ is not assumed to satisfy itself such an inequality.

When $\nu$ has compact support, included in the euclidean ball $B(0,R)$, this question has been partly studied in \cite{Zim}, and the results therein extended in \cite{BGMZ}. \cite{Zim} is using the Lyapunov function method of \cite{CGWPTRF}, while \cite{BGMZ} is partly using the Bakry-Emery criterion. Indeed since for $\sigma >0$,  $\nu^\sigma(dx)= e^{-V^\sigma}(x) \, dx$ for some smooth $V^\sigma$, with $$V^\sigma(x) = - \, \ln \left(\int e^{- \, \frac{|x-y|^2}{2 \sigma^2}} \, (2\pi \sigma^2)^{-n/2} \, \nu(dy)\right) \, .$$ A simple calculation (see \cite{BGMZ} p.438) shows that
\begin{equation}\label{bakemmollif}
Hess V^\sigma \, \geq \, \left(\frac{1}{\sigma^2} - \frac{R^2}{\sigma^4}\right) \, Id \, ,
\end{equation}
so that $C_{LS}(\nu^\sigma) \leq \frac{2\sigma^4}{\sigma^2 - R^2}$ as soon as $\sigma > R$ (it seems that the factor 2 is lacking in \cite{BGMZ}). The small variance case is more delicate and impose to use other arguments. Nevertheless it is not difficult using a variance decomposition to prove that the following is always true: $$C_P(\nu^\sigma) \, \leq \, \sigma^2 \, e^{4R^2/\sigma^2} \, .$$ 

If $V^\sigma$ is not necessarily strongly convex, the Hessian remains bounded from below. Using deep results by E. Milman (\cite{emil2}) in the spirit of the ones we will recall in the next section, it is shown in Theorem 4.3 of \cite{BGMZ} that $\nu^\sigma$ is still satisfying a log-Sobolev inequality that does not depend on the dimension $n$ provided $\sigma >R/\sqrt 2$, but this time the log-Sobolev constant is not explicit. 
\medskip

We will improve the latter result and furnish an explicit constant by directly using our perturbation results. To this end we simply write
\begin{equation}\label{eqpertcomp}
\nu^\sigma(dx) = Z^{-1} \, e^{-F(x)} \, \gamma_\sigma(dx) \quad \textrm{ with } \quad F(x)=V^\sigma(x) \, - \, \frac{|x|^2}{2 \sigma^2} \, .
\end{equation}
We have
\begin{equation}\label{eqpertcomp2}
\nabla F(x) = \int \, \frac{x-y}{\sigma^2} \, h(x,y) \, \nu(dy) \; - \; \frac{x}{\sigma^2}= \, - \, \int \, \frac{y}{\sigma^2} \, h(x,y) \, \nu(dy)
\end{equation}
where $$h(x,y) = \frac{e^{- \, \frac{|x-y|^2}{2 \sigma^2}}}{\int \, e^{- \, \frac{|x-z|^2}{2 \sigma^2}} \, \nu(dz)} \, .$$ Hence
\begin{equation}\label{eqpertcomp3}
|\nabla F(x)| \, \leq \, \frac{1}{\sigma^2} \, \int \, |y| \, h(x,y) \, \nu(dy) \, \leq \, \frac{R}{\sigma^2} \, .
\end{equation}
It remains to apply Theorem \ref{thmperturbnaivegene1}, part (2) of Theorem \ref{thmnaivelogsob} and the bounds $C_P(\gamma_\sigma) \leq \sigma^2$ and $C_{LS}(\gamma_\sigma) \leq 2 \, \sigma^2$ in order to get
\begin{theorem}\label{thmconvolcompact}
Let $\nu$ be any probability measure whose support is included in $B(0,R)$. Define $\nu^\sigma=\nu * \gamma_\sigma$ where $\gamma_\sigma$ denotes the centered gaussian distribution with covariance matrix $\sigma^2 \, Id$. Then if $$s:=\frac{1+\varepsilon}{4} \, \frac{R^2}{\sigma^2} < 1$$ it holds $$C_P(\nu^\sigma) \, \leq \, \frac{1+\varepsilon^{-1}}{1-s} \, \sigma^2 \, .$$ Similarly for all $\theta$ and $\beta$ positive, $$C_{LS}(\nu^\sigma) \leq \, \left(\frac{2(\beta +1)(1+\theta^{-1}) }{\beta} \, + \, 5 \, \frac{1+\varepsilon^{-1}}{1-s}\right) \, \sigma^2 \, $$ $$ \qquad \qquad + \, 2  \, \frac{1+\varepsilon^{-1}}{1-s} \, \left(\frac{(1+\theta)(1+\beta)}{4\beta} + \frac{\beta^2}{2}\right) \, R^2 \, .$$  
\end{theorem}
Notice that this result covers the range $\sigma >R/2$ which is larger than the one in \cite{BGMZ}. Here we have used $\mu(F) \leq F(0) + L \, \mu(|x|)$ in order to simplify the (already intricate) bound for the log-Sobolev constant. Using the elementary general $C_P(\mu*\nu) \leq C_P(\mu)+C_P(\nu)$, one has $C_P(\nu^\sigma) \leq C_P(\nu^{\sigma_0}) + (\sigma-\sigma_0)^2$, yielding the correct asymptotic behaviour. 
\medskip

The previous proof can easily be extended to more general situations replacing $\gamma_\sigma$ by some more general $\mu(dx)=e^{-H(x)} \, dx$, provided $Hess H$ is bounded, yielding
\begin{theorem}\label{thmconvolcompact2}
Let $\nu$ be any probability measure whose support is included in $B(0,R)$. Define $\nu^H=\nu * \mu$ where $\mu(dx)=e^{-H(x)} \, dx$ is a probability measure such that $$\sup_x |Hess H(x)| = K < +\infty \, .$$ Then if $$s:=\frac{1+\varepsilon}{4} \, K^2 \, R^2 \, C_P(\mu) < 1$$ it holds $$C_P(\nu^H) \, \leq \, \frac{1+\varepsilon^{-1}}{1-s} \, C_P(\mu) \, .$$
\end{theorem}
\begin{proof}
Following the notations of the previous proof we have $$\nabla F(x) = \int \, \nabla H(x-y) \, h(x,y) \, \nu(dy) \; - \; \nabla H(x) = \int \, (\nabla H(x-y)-\nabla H(x)) \, h(x,y) \, \nu(dy)$$ with $$h(x,y) = \frac{e^{- \, H(x-y)}}{\int \, e^{- \, H(x-z)} \, \nu(dz)} \, .$$ It remains to use $$|\nabla H(x-y)-\nabla H(x)| \, \leq \, K \, |y| \, ,$$ and to use Theorem \ref{thmperturbnaivegene1}.
\end{proof}
\begin{corollary}\label{corconvolcompact2}
Let $X$ be a random variable supported by $B(0,R)$ and $Y$ a random variable with distribution $\mu(dx)=e^{-H(x)} \, dx$ such that $\sup_x |Hess H(x)| = K < +\infty$. For $\sigma \in \mathbb R^+$ define $X^\sigma= X + \sigma Y$ and denote by $\nu^{\sigma H}$ the distribution of $X^\sigma$. Then if $$s:=\left(\frac{1+\varepsilon}{4} \, K^2 \, R^2 \, C_P(\mu)\right)/\sigma^2 < 1$$ it holds $$C_P(\nu^H) \, \leq \, \frac{1+\varepsilon^{-1}}{1-s} \, C_P(\mu) \, \sigma^2 \, .$$
\end{corollary}
\begin{proof}
It is enough to remark that the probability density of $\sigma Z$ is proportional to $e^{-H(x/\sigma)}$ so that $|\nabla F| \leq \frac{KR}{\sigma^2}$ and to remember that $C_P(\sigma Y)= \sigma^2 \, C_P(Y)$.
\end{proof}
\begin{remark}\label{remopti}
One can ask about what happens when $\nabla H$ is bounded, for instance if $\mu(dx)=Z^{-1} \, e^{- \sigma |x|} \, dx$ in $\mathbb R$. The proofs above furnish $|\nabla F|\leq 2/\sigma$ for all $R$ and all $\sigma$ so that the condition on $s$ reads $s:=(1+\varepsilon) \, C_P(\mu) < 1$ which is impossible since $C_P(\mu)=4$. This is another argument showing that our perturbation result is close to be optimal. \hfill $\diamondsuit$
\end{remark}
\medskip

\section{Perturbation with log-concavity.}\label{secperturb}

We will now give some new results relating the Poincar\'e constant of both measures $\mu$ and $\mu_F$ when at least one of them is log-concave. 
\medskip

For log-concave distributions it is often better to  use the Cheeger constant instead of the Poincar\'e constant. Recall the following
\begin{proposition}\label{propcheeg}
Recall that in all cases $C_P(\mu) \leq 4 \, (C'_C)^2(\mu) \leq 4 \, C_C^2(\mu)$. If in addition $\mu$ is log-concave the following converse inequality is satisfied :$$C'_C(\mu) \, \leq  \, C_C(\mu) \, \leq \, \frac{16}{\pi} \, \sqrt{C_P(\mu)} \, .$$ 
\end{proposition}
The first inequality is contained in \cite{bob99}, while the second one is shown in \cite{CGlogconc} proposition 9.2.11. With the slightly worse constant $6$ instead of $\frac{16}{\pi}$ the result is due to Ledoux in \cite{ledgap}. 
\medskip

Finally a remarkable property of log-concave measures, we shall intensively use in the sequel, is that a very weak form of the Poincar\'e (or Cheeger) inequality is enough to imply the true one. For simplicity we recall here the two main results we obtained in \cite{CGlogconc} (Theorem 9.2.7 and Theorem 9.2.14), improving on the beautiful seminal result by E. Milman (\cite{emil1})

\begin{theorem}\label{thmlogconc}
Let $\nu$ be a log-concave probability measure. 
\begin{enumerate}
\item[(1)] \; Assume that there exists some $0\leq s<1/2$ and some $\beta(s)$ such that for any Lipschitz function $f$ it holds $$\nu(|f-m_{\nu}(f)|) \leq \beta(s) \, \parallel |\nabla f|\parallel_\infty + s \, \Osc(f) \, .$$ Then $$C'_C(\nu) \leq \frac{4 \beta(s)}{\pi \, (\frac 12 - s)^2} \, .$$
\item[(2)] \; Assume that there exists some $0\leq s<1/6$ and some $\beta(s)$ such that for any Lipschitz function $f$ it holds $$\Var_\nu(f) \leq \beta(s) \, \nu(|\nabla f|^2) + s \, \Osc^2(f) \, .$$ Then $$C'_C(\nu) \leq \frac{4 \sqrt{\beta(s) \ln2}}{1-6s} \, .$$
\end{enumerate}
In both cases recall that $C_P(\nu) \leq 4 (C'_C(\nu))^2$.
\end{theorem} 
\smallskip

\subsection{From Holley-Stroock to Barthe-Milman. \\ \\}\label{subsecholley}

\noindent We start by mimiking the proof of Holley-Stroock perturbation result. Let $f$ be a bounded Lipschitz function. Of course in the definition of $\mu_F$ we may always replace $e^{-F}$ by $e^{-(F-\min F)}$ provided $F$ is bounded from below. Hence, for simplicity we may first assume that $F\geq 0$ so that $e^{-F} \leq 1$ is in all the $\mathbb L^p(\mu)$. Then :
\begin{eqnarray*}
\mu_F(|f-\mu_F(f)|) &\leq& 2 \, \mu_F(|f-m_{\mu_F}(f)|) \, \leq \, 2 \, \mu_F(|f-\mu(f)|) \\ &\leq& 2 \; \frac{\mu\left(|f-\mu(f)| \, e^{-F}\right)}{\mu(e^{-F})} \\ &\leq& \, \frac{2}{\mu(e^{-F})} \; \mu^{1/2}(|f-\mu(f)|^2) \; \mu^{1/2}(e^{-2F}) \\ &\leq& \frac{2 \, \mu^{1/2}(e^{-2F})}{\mu(e^{-F})} \, C_P^{1/2}(\mu) \, \mu^{1/2}(|\nabla f|^2) \, \leq \, \frac{2 \, \mu^{1/2}(e^{-2F})}{\mu(e^{-F})} \, C_P^{1/2}(\mu) \,  \parallel|\nabla f|\parallel_\infty \, . 
\end{eqnarray*}
If $F$ is not bounded from below, just using a cut-off we obtain the same result, with a possibly infinite right hand side. Using Theorem \ref{thmlogconc} (1), we can thus conclude 
\begin{proposition}\label{thmperths}
If $\mu_F$ is log-concave then $$C'_C(\mu_F) \, \leq \, \frac{32}{\pi} \, \frac{\mu^{1/2}\left(e^{-2F}\right)}{ \mu\left(e^{-F}\right)} \, C_P^{1/2}(\mu)$$ so that $$C_P(\mu_F) \, \leq \, \frac{4 \times 32^2}{\pi^2} \; \frac{\mu\left(e^{-2F}\right)}{\mu^2\left(e^{-F}\right)} \; C_P(\mu) \, .$$
\end{proposition}
\medskip

If interesting in comparison with Holley-Stroock the previous result is far from optimal. As shown in Theorem 2.7 of \cite{barmil}, there exists an universal constant $c$ such that $$C_P(\mu_F) \, \leq \, c \; \left(1 \, + \, \ln\left(\frac{\mu^{\frac 12}\left(e^{-2F}\right)}{\mu\left(e^{-F}\right)}\right)\right)^2 \; C_P(\mu) \, .$$ We shall recover this result and furnish a numerical bound for $c$. To this end first recall 

\begin{definition}\label{defconcent}
 The concentration profile of a probability measure $\nu$ denoted by $\alpha_\nu$, is defined as $$\alpha_\nu(r) := \sup \left\{1 - \, \nu(A+B(0,r)) \, ; \, \nu(A) \geq \frac 12\right\} \, , \, r>0 \, ,$$ where $B(y,r)$ denotes the euclidean ball centered at $y$ with radius $r$.
\end{definition}
The following is shown in \cite{CGlogconc} Corollary 9.2.10
\begin{proposition}\label{propconcentlogconc}
For any log-concave probability measure $\nu$, $$C'_C(\nu) \leq \inf_{0<s<\frac 14} \, \frac{16 \, \alpha_\nu^{-1}(s)}{\pi \, (1-4s)^2} \quad \textrm{ and } \quad C_P(\nu) \leq  \inf_{0<s<\frac 14} \, \left(\frac{32 \, \alpha_\nu^{-1}(s)}{\pi \, (1-4s)^2}\right)^2 \, .$$
\end{proposition}
Actually a better result, namely $$C'_C(\nu) \leq \, \frac{\alpha_\nu^{-1}(s)}{1-2s}$$ that holds for all $s<\frac 12$ was shown by E. Milman in Theorem 2.1 of \cite{Emil3}, when $\nu$ is the uniform measure on a convex body. The results extends presumably to any log-concave measure, but the proof of this result lies on deep geometric results (like the Heintze-Karcher theorem) while ours is elementary. 

Now recall the statement of Proposition 2.2 in \cite{barmil}, in a simplified form: if $M=\frac{\mu^{\frac 12}\left(e^{-2F}\right)}{\mu\left(e^{-F}\right)}$, then $$\alpha_{\mu_F} \, \leq \, 2M \, \alpha^{1/2}_\mu(r/2) \, .$$ We may use this result together with proposition \ref{propconcentlogconc} to deduce corollary 9.3.2 in \cite{CGlogconc} 
\begin{corollary}\label{corconc1}
If $\mu_F$ is log-concave, denoting $M=\frac{\mu^{\frac 12}\left(e^{-2F}\right)}{\mu\left(e^{-F}\right)}$, $$C'_C(\mu_F) \leq \, \inf_{0<s<\frac 14} \, \frac{32 \, \alpha_\mu^{-1}((s/2M)^2)}{\pi(1-4s)^2} \, .$$ 
\end{corollary}
Finally, as shown by Gromov and V. Milman, the concentration profile of a measure with a finite Poincar\'e constant is exponentially decaying. One more time it is not as easy to find an explicit version of Gromov-Milman's result. We found two of them in the literature: the first one in \cite{Troy} Th\'eor\`eme 25 (in french) and Proposition 11 $$\alpha_\mu(r) \leq 16 \, e^{- \, \frac{r}{\sqrt{2 \, C_P(\mu)}}} \, ,$$ the second one in \cite{Beres} Theorem 2: $$\alpha_\mu(r) \leq e^{- \, \frac{r}{3 \, \sqrt{C_P(\mu)}}} \, .$$ We shall use the first one due to the lower constant and obtain
\begin{theorem}\label{thmcorconc1}
If $\mu_F$ is log-concave, denoting $M=\frac{\mu^{\frac 12}\left(e^{-2F}\right)}{\mu\left(e^{-F}\right)}$, for all $0<s<\frac 14$, $$C_P(\mu_F) \, \leq \frac{C}{(1-4s)^4} \, \left(3 \ln 2+ \ln (1/s) + \ln M \right)^2 \, C_P(\mu)$$ where the constant $C$ satisfies $C \leq \left(\frac{64 \, \sqrt 2}{\pi}\right)^2 \, .$
\end{theorem}
\medskip

\begin{remark}\label{rembmil}
Of course $M\geq 1$. In order to get a presumably more tractable bound, first remark that adding a constant to $F$ does not change $M$ so that we may always assume that $\min F=0$. It thus follows $\mu(e^{-2F}) \leq \mu(e^{-F})$. According to Jensen's inequality we also have $\mu(e^{-F}) \geq e^{-\mu(F)}$ so that finally $$M \leq e^{\frac 12 \, \mu(F)} \, .$$ Finally $$C_P(\mu_F) \leq (C_1 + C_2 \, \mu^2(F)) \, C_P(\mu) \, $$ for some explicit constants $C_1$ and $C_2$.
\hfill $\diamondsuit$
\end{remark}
\medskip

\begin{example}{\textbf{Gaussian perturbation.}}\label{examplegauss}

For $\rho \in \mathbb R^+$ consider $$\mu^\rho(dx) = Z_\rho^{-1} \, e^{-V(x) \, - \, \frac 12 \, \rho \, |x|^2} \, dx \, ,$$ i.e. $$F(x) = \frac{\rho}{2} \, |x|^2 \, .$$ If $V$ is convex (hence $\mu^\rho$ log-concave), we deduce from the previous theorem and Bakry-Emery criterion that,  $$C_P(\mu^\rho) \, \leq  \, \min \; \left(\frac 1\rho \; ; \; \left(C_1+ C_2 \, \ln^2\left(\frac{\mu^{\frac 12}(e^{- \, \rho \, |x|^2})}{\mu(e^{- \, \rho \, |x|^2/2})}\right)\right) \, C_P(\mu)\right) \, ,$$ for some explicit universal constants $C_1$ and $C_2$. 

This indicates that we can find a bound for the Poincar\'e constant of $\mu^\rho$ that does not depend on $\rho$. We shall try to get some explicit result.
\medskip

First according to remark \ref{rembmil}, the $\ln^2$ can be bounded up to some universal constant by $\mu^2(\rho \, |x|^2)$ so that if $\mu$ is isotropic i.e. is centered with a covariance matrix equal to identity, we obtain a bound in $\rho^2 \, n^2$. Optimizing in $\rho$, we see that the worst case  is for $\rho \sim n^{- \, \frac 23} \, C_P^{- \, \frac 13}(\mu)$ yielding 
\begin{proposition}\label{propgaussperturb}
If $\mu(dx) = Z^{-1} \, e^{-V(x)} \, dx$ is log-concave and isotropic, then for all $\rho\geq 0$, and $\mu^\rho(dx) = Z_\rho^{-1} \, e^{-V(x) \, - \, \frac 12 \, \rho \, |x|^2} \, dx$, $$C_P(\mu^\rho) \leq C \, n^{2/3}  \, C_P^{\frac 13}(\mu) \, ,$$ for some universal constant $C$.
\end{proposition}

This result looks disappointing. A direct approach via the KLS inequality obtained by Chen will presumably give a better dimensional bound provided we are able to get a good bound for the covariance matrix of the perturbed measure. When $V$ is even we can directly estimate the covariance matrix of $\mu^\rho$ as in Theorem 18 in \cite{BK19}.
\hfill $\diamondsuit$
\end{example}

\medskip

\subsection{Smooth perturbations.\\ \\}\label{subsecnaive}

Using the specific properties of log-concave measures we can state a first result

\begin{theorem}\label{thmperturbnaivelogconc1}
If $\mu_F$ is log-concave we have for $\varepsilon>0$, $$C_P(\mu_F) \, \leq \, \frac{64 \, \ln(2) \, (1+\varepsilon^{-1}) \, C_P(\mu)}{(1-6s)^2} \quad \textrm{provided} \quad \frac{1+\varepsilon}{4} \, C_P(\mu) \, \mu_F(|\nabla F|^2):=s \, < \, \frac 16 \, .$$ 
\end{theorem}
In particular this bound is available as soon as $F$ is $L$-Lipschitz with $L^2< 2/(3 \,(1+\varepsilon) \, C_P(\mu))$, but this condition is worse than the one in theorem \ref{thmperturbnaivegene1}.
\smallskip
\begin{proof}
We deduce from \eqref{eqnaive1bis}
\begin{equation}\label{eqnaive2}
\Var_{\mu_F}(f) \, \leq \, (1+\varepsilon^{-1}) \, C_P(\mu) \, \mu_F(|\nabla f|^2) \, + \, \frac{1+\varepsilon}{4} \, C_P(\mu) \, \Osc^2(f) \, \mu_F(|\nabla F|^2) \, .
\end{equation} 
So that, if $\mu_F$ is log-concave,  using Theorem \ref{thmlogconc} (2) we get the result.
\end{proof}

We may similarly modify the proof of proposition \ref{propperturbnaivecheeger} to similarly get a Cheeger inequality.

\begin{theorem}\label{thmperturbnaivelogconccheeger}
If $\mu_F$ is log-concave we have $$C'_C(\mu_F) \, \leq \, \frac{16 \, C_C(\mu)}{\pi (1-2s)^2} \quad \textrm{provided} \quad  C_C(\mu) \, \mu_F(|\nabla F|):=s \, < \, \frac 12 \, .$$ In particular if $\mu$ is also log-concave we have $$C_P(\mu_F) \, \leq \, \frac{256 \times 64}{\pi^4} \, \frac{C_P(\mu)}{(1-2s)^4} \, ,$$ for $s$ as before.
\end{theorem}
At the level of Cheeger inequality, we found no other comparable perturbation result despite, once again, the very simple argument involved here.
\medskip

Starting with \eqref{eqnaive3bis} we also have $$\Var_{\mu_F}(f) \leq \mu_F((f-a)^2) \leq C_P(\mu) \mu_F(|\nabla f|^2) + \frac 12 \, C_P(\mu) \, \Osc^2(f) \, \mu_F([AF - \frac 12 \, |\nabla F|^2]_+)$$ so that we obtain an improvement of Theorem \ref{thmperturbnaivegene2}
\begin{theorem}\label{thmperturbnaivelogconc2}
If $V$ is $C^1$, $\mu_F$ is log-concave, $F$ is of $C^2$ class and satisfies $$C_P(\mu) \, \mu_F([AF - \frac 12 \, |\nabla F|^2]_+):= s \, < \, \frac 13$$ then $$C_P(\mu_F) \, \leq \, \frac{64 \, \ln(2)}{(1-3s)^2} \; C_P(\mu) \, .$$ 
\end{theorem}
\medskip

\begin{remark}\label{remsupport}
As in the previous section it is interesting to extend the result to compactly supported log-concave measures. We will thus assume that the set $U=\{V<+\infty\}$ is convex with a smooth boundary and that $V$ is $C^1$ in $U$. Then the conclusion of the previous Theorem is still available provided in addition $\partial_n F \geq 0$ on $\partial U$. \hfill $\diamondsuit$
\end{remark} 
\medskip

\begin{remark}\label{remgamma2bounded}
Assume that $V$ is $C^2$ on $\mathbb R^n$ and satisfies $\langle u,Hess_V(x) u\rangle \geq \rho |u|^2$ for all $u$ and $x$ in $\mathbb R^n$. Let $U$ be an open convex subset given by $U=\{W <1\}$ where $W$ is a smooth (say $C^2$) convex function. Consider $$\mu(dx) \, = \, Z^{-1} \, e^{-V(x)} \, \mathbf 1_{W(x)\leq 1} \; dx \, .$$ It turns out that once again $$C_P(\mu) \leq 1/\rho \, .$$ To prove this result one can use the $\Gamma_2$ theory of Bakry-Emery, but one has to carefully define the algebra $\mathcal A$ (see section 1.16 in \cite{BaGLbook}). The devil is in this definition when looking at reflected semi-groups. We prefer to give an elementary proof of what we claimed.

Define $H(x) = (W(x)-1)^4 \, \mathbf 1_{W(x)\geq 1}$. $H$ is smooth and $$\partial_{ij}^2 H (x) = 4 (W(x)-1)^2 \, \left((W-1) \, \partial_{ij}^2 W + 3 \, \partial_i W \, \partial_j W\right) \, \mathbf 1_{W(x)\geq 1}$$ so that $$\langle u, Hess_H(x) u\rangle =  4 (W(x)-1)^2 \, \left((W-1) \, \langle u, Hess_W(x) u\rangle + 3 \, \langle u,\nabla W\rangle^2\right) \, \mathbf 1_{W(x)\geq 1} \, \geq \, 0 \, .$$ If we consider $$\mu^\varepsilon(dx) = Z_\varepsilon^{-1} \, e^{-V(x) - \frac 1\varepsilon \, H(x)} \, dx \, ,$$ $\mu^\varepsilon$ satisfies the Bakry-Emery criterion so that $C_P(\mu^\varepsilon) \leq 1/\rho$. It remains to let $\varepsilon$ go to $0$ and to pass to the limit in \eqref{eqpoinc} by using Lebesgue convergence theorem. \hfill $\diamondsuit$
\end{remark}
\medskip

Let us finish by proving the same type of result at the level of Brascamp-Lieb inequality. Let us consider $d\mu=e^{-V}dx$ with $\mbox{Hess}(V)>0$ in the sense of positive definite matrix, the celebrated Brascamp-Lieb inequality is then
$$\Var_{\mu}(f)\,\le\, \int (\nabla f)^t \,\mbox{Hess}(V)^{-1}\,\nabla f d\mu.$$
We will see that we can easily have some perturbation result, leading to some modified Brascamp-Lieb inequality.
\begin{theorem}
Let us consider $d\mu=e^{-V}dx$ with $\mbox{Hess}(V)>0$ in the sense of positive definite matrix and suppose that there exists $\epsilon$ such that
$$\frac14(1+\epsilon)\|(\nabla F)^t\,\mbox{Hess}(V)^{-1}\,\nabla F\|_\infty<1$$
then
$$\Var_{\mu_F}(f)\,\le\, \frac{(1+\epsilon^{-1})}{1-\frac14(1+\epsilon)\|(\nabla F)^t\,\mbox{Hess}(V)^{-1}\,\nabla F\|_\infty}\,\int \nabla f^t \mbox{Hess}(V)^{-1}\nabla f d\mu \, .$$ In particular if $\mu_F$ is log-concave then $$C_P(\mu_F) \leq \, \frac{64 \, (1+\epsilon^{-1}) \, \ln(2)}{1-\frac14(1+\epsilon)\|(\nabla F)^t\,\mbox{Hess}(V)^{-1}\,\nabla F\|_\infty} \, \int \parallel \mbox{Hess}(V)^{-1}\parallel_{HS} \, d\mu $$ where $\parallel .\parallel_{HS}$ denotes the Hilbert-Schmidt norm.
\end{theorem}
\begin{proof}
We follow the idea already developed for the Poincar\'e inequality. 
\begin{equation*}
\Var_{\mu_F}(f) \, \leq \, \mu_F((f-a)^2) \, = \,  \mu^{-1}(e^{-F}) \, \mu\left(((f-a) e^{- \frac 12 \, F})^2\right) \, 
\end{equation*}
 for which we choose $$a=\frac{\mu\left(f \, e^{- \frac 12 \, F}\right)}{\mu\left(e^{- \frac 12 \, F}\right)}$$ so that we may apply Brascamp-Lieb inequality for $\mu$.
 \begin{eqnarray*}
 \mu_F((f-a)^2)\, &\leq&  \, \int \, (\nabla f \, - \, \frac 12 \, (f-a) \, \nabla F)^t\mbox{Hess}(V)^{-1} \, (\nabla f \, - \, \frac 12 \, (f-a) \, \nabla F)\,d\mu_F \nonumber \\ 
 &\leq& \,  (1+\epsilon^{-1})\int(\nabla f)^t\mbox{Hess}(V)^{-1}\nabla fd\mu_F \\&&\qquad+ \frac14(1+\epsilon) \, \int(f-a)^2 \, (\nabla F)^t\mbox{Hess}(V)^{-1}\nabla F\,d\mu_F.
 \end{eqnarray*}
 We then use our growth condition to conclude.
 
If $\mu_F$ is log-concave, we have $$\Var_{\mu_F}(f)\,\le\, \frac{(1+\epsilon^{-1})}{1-\frac14(1+\epsilon)\|(\nabla F)^t\,\mbox{Hess}(V)^{-1}\,\nabla F\|_\infty} \, \parallel |\nabla f|\parallel_\infty^2 \,  \int \parallel \mbox{Hess}(V)^{-1}\parallel_{HS} \, d\mu \, ,$$ and we can conclude by using Theorem \ref{thmlogconc} (2).
\end{proof}

Of course it is illusory to expect a Brascamp-Lieb inequality for $\mu_F$ as $\mbox{Hess}(V+F)$ is not necessarily positive. However it may be useful for concentration inequalities, indeed, reproducing the proof of the exponential integrability for Poincar\'e inequality due to Bobkov-Ledoux, see \cite{BaGLbook}, under the assumptions of the previous theorem, if $f$ is such that
$$\|(\nabla f)^t\,\mbox{Hess}(V)^{-1}\,\nabla f\|_\infty\le 1$$
then 
$$\forall s<\sqrt{\frac{4\left(1-\frac14(1+\epsilon)\|(\nabla F)^t\,\mbox{Hess}(V)^{-1}\,\nabla F\|_\infty\right)}{1+\epsilon^{-1}}},\qquad \int e^{sf}d\mu_F<\infty.$$
This thus implies exponential concentration for $\mu_F$ for some particular class of functions.
\medskip

\section{Coming back: from the perturbed measure to the initial one.}\label{seckls}

Any probability measure $\mu(dx)=e^{-V(x)} dx$ can be seen as a perturbation of a perturbed measure, namely $\mu(dx)=Z^{-1} \, e^{F(x)} \, \mu_F(dx)$. In some cases the measure $\mu_F$ is simpler to study, so that one can expect some results for the initial one using our perturbation method.

\subsection{Some consequences using gaussian perturbation.\\ \\}\label{subsecgauss2}

As an immediate application, for $\rho \in \mathbb R^+$ consider $$\nu(dx) = Z_\rho^{-1} \, e^{-V(x) \, - \, \frac 12 \, \rho \, |x|^2} \, dx .$$ If we denote $$F(x)= - \, \frac 12 \, \rho \, |x|^2$$ we have with the previous notations $$\mu(dx) = \nu_F(dx) \, .$$ If $V$ is convex, $\nu$ satisfies the Bakry-Emery criterion and accordingly $C_P(\nu) \leq 1/\rho$. We shall use the results in the previous subsection, starting with Theorem \ref{thmperturbnaivelogconc1} with $\varepsilon=1$ for simplicity. 

Hence $$\frac 12 \, C_P(\nu) \, \nu_F(|\nabla F|^2) \, \leq \, \frac{1}{2\rho} \, \rho^2 \, \mu(|x|^2) \, .$$ In order to apply Theorem \ref{thmperturbnaivelogconc1} we thus need $s=\frac 12 \, \rho \, \mu(|x|^2)\leq 1/6$.  We have thus obtained, choosing $\rho$ small enough
$$C_P(\mu) \, \leq \, 128 \, \ln(2) \, \frac{1}{2s (1-6s)^2} \, \mu(|x|^2) \, .$$ The optimal choice of $s$ is $1/18$ and of course we may always center $\mu$ without changing the Poincar\'e constant.
\begin{corollary}\label{corgauss1}
Let $\mu$ be a log-concave measure. Then $$C_P(\mu) \, \leq \, 32 \times 81 \, \ln(2) \, \mu(|x-\mu(x)|^2) \, .$$
\end{corollary}
This result is well known and according to \cite{Alonbast} p.11 is contained in \cite{KLS} with a much better pre-constant $4$ (also see \cite{bob99} (1.8) with a non explicit constant). 

Applying Theorem \ref{thmlogconc} it is easily seen (see (9.2.13) in \cite{CGlogconc}) that 
\begin{equation}\label{eqcheeggafa}
C'_C(\mu) \leq \frac{16}{\pi} \, \mu(|x-m_\mu(x)|) \leq \frac{16}{\pi} \, \mu(|x-\mu(x)|) \, .
\end{equation}

If we replace Theorem \ref{thmperturbnaivelogconc1} by Theorem \ref{thmperturbnaivelogconccheeger} we obtain using our perturbation method the worse $$C'_C(\mu) \, \leq \, \frac{100 \, \sqrt{10}}{\pi^2}  \, \mu(|x-\mu(x)|) \, .$$ 

\begin{remark}\label{remhess1}
In what precedes we may replace $F(x)=-\frac 12 \, \rho \, |x|^2$ by $F(x) = - \rho \, H(x)$ with $Hess H(x) \geq Id$, without changing anything. Hence for example Corollary \ref{corgauss1} can be generalized in: 
\begin{corollary}\label{corhess1}
If $\mu$ is log-concave, for all $C^2$ function $H$ satisfying $Hess H(x) \geq Id$, it holds
\begin{equation}\label{eqhess1}
C_P(\mu) \, \leq \, 32 \times 81 \, \ln(2) \, \mu(|\nabla H|^2) \, .
\end{equation}
\end{corollary}
\hfill $\diamondsuit$
\end{remark}
\medskip

\begin{remark}
What happens if instead we try to use Theorem \ref{thmperturbnaivelogconc2}. With the notations of the previous subsection it holds $$A=\Delta - \nabla V.\nabla - \rho \, x.\nabla$$ so that 
\begin{equation}\label{eqkls1}
AF \, - \, \frac 12 \, |\nabla F|^2 = - \, \rho n \, + \rho \, x.\nabla V(x) \, + \, \frac 12 \, \rho^2 \, |x|^2 \, ,
\end{equation}
and $$C_P(\mu) \, \leq \, \frac{64 \, \ln(2)}{(1-3s)^2 \, \rho}$$ as soon as $$\mu\left(\left[- \, \rho n \, + \rho \, x.\nabla V(x) \, + \, \frac 12 \, \rho^2 \, |x|^2\right]_+\right) \, \leq \, \frac{\rho}{3} \, .$$
Though the result looks stronger than the previous ones we did nod succeed in really exploring some interesting consequences, in terms of dimensional controls of the Poincar\'e inequality. 
\hfill $\diamondsuit$
\end{remark}
\medskip

\subsection{Using product Subbotin (exponential power) perturbations. \\ \\}\label{subsecsubbot}


Here we shall use the idea of the previous subsection replacing the gaussian measure by the tensor product of Subbotin distributions and the Bakry-Emery criterion by results of Barthe-Klartag.  

For $p \geq 1$ and $\lambda >0$ consider $$\nu(dx) = Z^{-1} \, e^{-V(x) - \lambda^p \, \sum_{i=1}^n \, |x_i|^p} \, dx \, ,$$ where we assume that $V$ is a convex function. Let $X$ be a random variable with distribution $\nu$, then $\lambda X$ has distribution $$\nu(\lambda,dx) = Z_\lambda^{-1} \,  e^{-V(x/\lambda) - \, \sum_{i=1}^n \, |x_i|^p} \, dx \, ,$$ and the dilation property for the Poincar\'e constant gives $$C_P(\nu) \, = \, \lambda^{-2} \, C_P(\nu(\lambda,.)) \, .$$ As before, if we denote $$F(x)= - \lambda^p \, \sum_{i=1}^n \, |x_i|^p$$ we have $$\mu(dx) = \nu_F(dx) \, .$$ We thus have 
\begin{equation}\label{eqnablasubot}
|\nabla F|^2(x) \, \leq \, \lambda^{2p} \, p^2 \, \sum_{i=1}^n \, |x_i|^{2(p-1)} \, ,
\end{equation}
so that 
\begin{equation}\label{eqnablasubotbis}
\mu(|\nabla F|^2)= \nu_F(|\nabla F|^2) \, \leq \,  \lambda^{2p} \, p^2 \,  \mu\left(\sum_{i=1}^n \, |x_i|^{2(p-1)}\right) \, .
\end{equation}
Choosing for simplicity $\varepsilon=1$ and $s=1/12$ in Theorem \ref{thmperturbnaivelogconc1} we thus have to choose (if this choice is possible)
\begin{equation}\label{eqchoixlambda}
\lambda^{2(p-1)} \, p^2 \, C_P(\nu(\lambda,.)) \, \mu\left(\sum_{i=1}^n \, |x_i|^{2(p-1)}\right) \, = \, \frac 16 \, .
\end{equation}
Notice that we may always use an upper bound for $C_P(\nu(\lambda,.))$ furnishing a lower bound for $\lambda$ and an upper bound for $C_P(\mu)$.
\medskip

For which $p$'s do we obtain interesting results ?
\medskip

One cannot expect that \eqref{eqchoixlambda} can be satisfied for $p=1$, since the left hand side is of size $n$. 
\medskip

If $p \neq 1$ we obtain, provided \eqref{eqchoixlambda} is satisfied,
\begin{equation}\label{poincsubot1}
C_P(\mu) \, \leq \, C \, 6^{\frac{1}{p-1}} \, p^{\frac{2}{p-1}} \, \mu^{\frac{1}{p-1}}\left(\sum_{i=1}^n \, |x_i|^{2(p-1)}\right) \, C_P^{\frac{p}{p-1}}(\nu(\lambda,.)) \, ,
\end{equation}
with $C=512 \, \ln(2)$.

In particular if $V$ is even, we may apply Theorem \ref{thmbarklar} (since $C_P(\nu(\lambda,.))$ can be bounded independently of $\lambda$) and get for $1<p<2$,
\begin{equation}\label{eqpoincsubot2}
C_P(\mu) \, \leq \, C \, 6^{\frac{1}{p-1}} \, p^{\frac{2}{p-1}} \, \mu^{\frac{1}{p-1}}\left(\sum_{i=1}^n \, |x_i|^{2(p-1)}\right) \, (\ln(n))^{\frac{2-p}{p-1}} \, ,
\end{equation}
for some universal $C$. 

Compared with \eqref{eqcheeggafa}, this result for $p=\frac 32$ is however bad w.r.t. the dimension.
\medskip

Looking at \eqref{eqpoincsubot2} it seems interesting to get an analogue of Theorem \ref{thmbarklar} i.e. a bound for $C_P(\nu(\lambda,.))$ that does not depend on $\lambda$ but for $p>2$. 

Indeed, if $V$ is even, $\mu(x)=0$ and for $2<p$, 
\begin{equation}\label{eqkhinch}
\mu(|x_i|^{2(p-1)}) \, \leq \, \frac{\Gamma(2p+1)}{2^{p-1}} \; \mu^{p-1}(|x_i|^2) \, \leq \, (p-1)^{2(p-1)} \, \mu^{p-1}(|x_i|^2) 
\end{equation}
 according to \cite{guedonpol} corollary 5.7 and remark 5.8 (also see in \cite{latala} the discussion after definition 2). In particular  we obtain, provided \eqref{eqchoixlambda} i.e.
\begin{equation}\label{eqchoixlambda2}
\lambda^2 = \left( \frac{1}{6 p^2 \, C_P(\nu(\lambda,.)) \, n}  \right)^{\frac{1}{p-1}} \, \frac{1}{(p-1)^2 \, \sigma^2(\mu)}
\end{equation}
is satisfied, for $p>2$,
\begin{equation}\label{poincsubot3}
C_P(\mu) \, \leq \, C \, 6^{\frac{1}{p-1}} \, p^{\frac{2}{p-1}} \, (p-1)^2 \, n^{\frac{1}{p-1}} \, C_P^{\frac{p}{p-1}}(\nu(\lambda,.)) \, \sigma^2(\mu) \, ,
\end{equation}
with $C= 512 \, \ln(2)$. This time if $p$ is of order $\ln(n)$ we will get an interesting result, provided $C_P(\nu(\lambda,.))$ is controlled by some not too bad constant (possibly dependeing on $n$). 

Unfortunately, \cite{BK19} contains an example (see subsection 3.4) where $C_P^{\frac{p}{p-1}}(\nu(\lambda,.))$ behaves like $n^{\frac{p-2}{p-1}}$, but for a measure $\mu$ which is highly non isotropic. 
\medskip

Nevertheless if we assume in addition that $\mu$ is unconditional (i.e $V(x_1,...,x_n)=V(|x_1|,...,|x_n|)$ for all $x$ so that $V(./\lambda)$ is also unconditional), it follows from Theorem 17 in \cite{BK19} that $$C_P(\nu(\lambda,dx)) \, \leq \, C_P(S_p(dx_1))$$ for all $\lambda$, where $$S_p(dx_1)= \frac{1}{z_p} \, e^{-|x_1|^p} \, .$$ According to Bobkov's one dimensional result (\cite{bob99} Corollary 4.3), 
\begin{equation}\label{eqbobsubbot}
C_P(S_p) \leq 12 \, \Var_{S_p}(x) \, = \, 12 \; \frac{\Gamma(3/p)}{\Gamma(1/p)} \, .
\end{equation}
A better result is obtained by combining \cite{BJMsubbot} Theorem 2.1  and the dilation property of Poincar\'e constants yielding $$C_P(S_p) \leq \frac{p^{1-2/p}}{2(1+p)^{1-2/p}} \, .$$ Hence in the unconditional case, 
\begin{eqnarray}\label{poincsubot4}
C_P(\mu) \, &\leq &\, C \, 6^{\frac{1}{p-1}} \, \left(\frac{p^2}{2}\right)^{\frac{1}{p-1}} \, (p-1)^2 \, n^{\frac{1}{p-1}} \, \frac{p^{\frac{p-2}{p-1}}}{(1+p)^{\frac{p-2}{p-1}}} \, \sigma^2(\mu) \, \nonumber\\ &\leq& \, 4C \, 3^{\frac{1}{p-1}} \, (p-1)^2 \, n^{\frac{1}{p-1}} \, \sigma^2(\mu) \, ,
\end{eqnarray}
with $C=512 \, \ln(2)$. Here we used $p^{2/(p-1)} \leq 4$ for $p\geq 2$.

It remains to optimize in $p$, the optimal value being $p-1 =\ln(3n)/2$ (which is larger than $1$ for $n\geq 2$. 
\begin{proposition}\label{propuncond}
For $n \geq 2$, any unconditional log-concave probability measure $\mu$ satisfies $$C_P(\mu) \, \leq \, C \, \ln^2(3n)) \, \sigma^2(\mu)$$ with $C= 512 \, e^2 \, \ln(2)$.  
\end{proposition}
This result is not new, and is due to Klartag in \cite{Klartuncond} with a non explicit constant. Another proof (still with a difficult to trace constant) is contained in \cite{CGlogconc}. The constant here is explicit (but certainly far to be sharp), but the most interesting fact is that this result can be obtained via Subbotin perturbation.
\begin{remark}\label{remsubbot}
Notice that the choice $p-1=\ln(3n)/2$ gives $\lambda \sim C/ln(n)$ for $n$ large enough, according to \eqref{eqchoixlambda2}.

In second place $C_p(S_p)$ is uniformly bounded in $p$ ($S_p$ weakly converges to the uniform distribution on $[-1,1]$ as $p$ goes to infinity). One can thus be tempted to use another product measure based on a one dimensional log concave family such that the Poincar\'e constant goes to $0$ as $p$ goes to infinity. We did not succeed in following this direction.\hfill $\diamondsuit$
\end{remark}
\begin{remark}\label{remredac}
We have only introduced $\nu(\lambda,dx)$ in order to directly connect the previous proof to Barthe and Klartag results. We should of course directly estimate the Poincar\'e constant of $\nu$. \hfill $\diamondsuit$
\end{remark}
\medskip

\section{Application to some problems in Bayesian statistics.}\label{secsparse}

\subsection{Sparse linear regression.\\ \\}\label{subsecsparse}

In \cite{DT12} the authors proposed a Bayesian strategy for the liner regression model 
\begin{equation}\label{eqewa1}
Y_i \, = \, \langle X_i \, , \, \lambda^*\rangle \, + \, \xi_i
\end{equation}
where $\lambda^*$ and each $X_i$ belong to $\mathbb R^M$, $\xi_i$ are i.i.d. scalar noises and $i=1,...,n$. $n$ is thus the size of a sample while $M$ is the dimension of the predictor. Given a collection of design points $X_i$ the exponentially weighted aggregate estimator of $\lambda^*$ is given by 
\begin{equation}\label{eqewa2}
\hat{\lambda}_n(X) = \int \, \lambda \, \hat{\pi}_{n,\beta}(d\lambda)
\end{equation}
where 
\begin{equation}\label{eqewa3}
\hat{\pi}_{n,\beta}(d\lambda) = C \, \exp\left(-(1/\beta) \sum_{i=1}^n \, |Y_i -\langle X_i,\lambda\rangle|^2\right) \, \pi(d\lambda)
\end{equation}
is the posterior probability distribution associated to the prior $\pi(d\lambda)=e^{-W(\lambda)} \, d\lambda$ and the temperature $\beta$. Here and in all what follows $C$ is a normalizing constant that can change from line to line.

In their Theorem 2 they obtain in particular an explicit bound for the $\mathbb L^2$ error, when the prior is (almost) chosen as 
\begin{equation}\label{eqewa4}
\pi(d\lambda) = C \; \prod_{j=1}^M \, \frac{e^{-\alpha \, |\lambda_j|}}{(\tau^2+\lambda_j^2)^2} \; \mathbf 1_{\sum_{j=1}^M |\lambda_j| \leq R} \; d\lambda \, = \, e^{-W(\lambda)} \, d\lambda ,
\end{equation}
for some positive $\alpha$ and $R$. This choice is motivated by dimensional reasons when $M\gg n$ and $\lambda^*$ is sparse. 
\medskip

In order to compute $\hat{\lambda}_n$, they propose to use the ergodic theorem applied to the Langevin diffusion process 
\begin{equation}\label{eqewa5}
dL_t = \sqrt 2 \, dB_t  - \nabla W(L_t) dt - \frac 2\beta \, (\langle X_i,L_t\rangle - Y_i) dt 
\end{equation}
where $B_.$ is a standard $\mathbb R^M$ valued Brownian motion, i.e. the $\mathbb L^1$ convergence of $\frac 1t \, \int_0^t \, L_s \, ds$ to the desired $\hat{\lambda}_n$ as $t \to +\infty$. The $\mathbf 1_{\sum_{j=1}^M |\lambda_j| \leq R}$ is no more considered here, and thus denote \begin{equation}\label{eqnusparse}
\nu_{n,\beta}(d\lambda) = C \, \exp\left(-(1/\beta) \sum_{i=1}^n \, |Y_i -\langle X_i,\lambda\rangle|^2-\sum_{i=1}^n\log(\tau^2+\lambda_i^2\right) \, \prod_{j=1}^M \, e^{-\alpha \, |\lambda_j|} \; d\lambda.
\end{equation}
Remark that $\nu_{n,\beta}$ is not logconcave. To justify some rate of convergence they call upon the Meyn-Tweedie theory (see e.g. \cite{MT3}) of Foster-Lyapunov functions. Notice that using instead $\frac 1t \, \int_t^{2t} \, L_s \, ds$ the bound (8) in \cite{DT12} becomes $C \, \theta^t$ for some $\theta<1$.

If the arguments give an exponential rate $\theta^t$ for some $\theta<1$ of convergence in their proposition 1, they are far to provide us with a bound for $-\ln(\theta)$. The reason is that constants are very difficult to trace with this theory. In several papers, \cite{BBCG,CGZ,CGhit}, we have established the links between the Poincar\'e inequality and the existence of Lyapunov functions, including their relationship with hitting times. One can find  explicit bounds in these papers.
A quick look to the explicit bounds in \cite{CGZ} show that it furnishes a rate of convergence that heavily depend on $M$, assumed to be big, on $\beta$ assumed to be big too, and the observations.
\medskip

Since $\nu_{n,\beta}$ is reversible for the Langevin diffusion, and satisfies a Poincar\'e inequality, $L_t$ itself converges to $\hat{\lambda}_n$ in $\mathbb L^2(\nu_{n,\beta})$ at an exponential rate given by $e^{-t/C_P(\nu_{n,\beta})}$. There is no need of Cesaro average. 

 
Our goal is thus to get some interesting bounds for $C_P(\nu_{n,\beta})$. To this end we may use two methods. In what follows $Z$ is a normalizing constant that may change from line to line.

First, we may write $\nu_{n,\beta}= \mu_F$ with $$\mu(d\lambda) = Z^{-1} \, \exp\left(-(1/\beta)\sum_{i=1}^n \, \langle X_i,\lambda\rangle^2\right) \;  \prod_{j=1}^M \, e^{-\alpha \, |\lambda_j|} \; d\lambda$$ and  
$$F(\lambda)= - \, \frac{2}{\beta}  \, \sum_{i=1}^n \, Y_i \, \langle X_i,\lambda\rangle-\sum_{i=1}^n\log(\tau^2+\lambda_i^2) \, = \, - \, \frac{2}{\beta}  \, \sum_{j=1}^M \, \lambda_j \, \langle Y,\tilde X^j\rangle -\sum_{i=1}^n\log(\tau^2+\lambda_i^2)$$where $\tilde X =(X_1^j,...,X_n^j) \, .$ But according to Barthe and Klartag result Theorem \ref{thmbarklar}, the tensorisation property and the quadratic behaviour of the Poincar\'e with respect to dilation we have $$C_P(\mu) \, \leq \, C \, \ln^2(M) \, \frac{1}{\alpha^2} \, ,$$ for some universal constant. Since $F$ is $L$-Lipschitz with $$L=\frac 2\beta \, \sup_{j=1,...,M} \, |\langle Y,\tilde X^j\rangle |+\frac1\tau$$ we may apply Theorem \ref{thmperturbnaivegene1}
\begin{theorem}\label{thmregress1}
There exist two universal positive constants $c$ and $C$ such that, provided $$\frac{\ln(M)}{\beta \, \alpha} \,\left( \sup_{j=1,...,M} \, |\langle Y,\tilde X^j\rangle | +\beta\tau\right)\,\leq c$$ then $$C_P(\nu_{n,\beta}) \, \leq \; C \, \frac{\beta}{\sup_{j=1,...,M} \, |\langle Y,\tilde X^j\rangle |+\beta\tau} \, .$$
\end{theorem}
Depending on whether the logarithmic factor $\ln(M)$ is necessary in Theorem \ref{thmbarklar} or not will yield a result that does not depend on the dimension.
\medskip

As said before the situation considered here is $M \gg n$. Another approach to compute $C_P(\mu)$ will thus be to use Theorem \ref{thmperths} instead of Theorem \ref{thmbarklar}. The discussion will thus be very similar to the one in Example \ref{examplegauss}. Of course because the quadratic form $\sum_{i=1}^n \, \langle X_i,\lambda\rangle^2$ is very degenerate on $\mathbb R^M$ it is impossible to use Bakry-Emery criterion. Since our goal is not to rewrite \cite{DT12} but to see how one can control the rate of convergence of the Langevin dynamics, we will assume for simplicity that the $X_i$'s are an orthogonal family. So, after an orthogonal transform, we may write $$\mu(d\lambda) = \, Z^{-1} \, \exp\left(-(1/\beta)\sum_{i=1}^n \, |X_i|^2 \, \lambda_i^2\right) \;  e^{-V(\lambda)} \; d\lambda$$ where $e^{-V(\lambda)} \; d\lambda$ an orthogonal change of $\prod_{j=1}^M \, e^{-\alpha \, |\lambda_j|} \; d\lambda$, hence shares the same Poincar\'e constant. Since the Poincar\'e constant of the symmetric exponential isotropic distribution in dimension $1$, hence the one of the tensor product of such distributions, is equal to $4$, the dilation scaling yields $C_P(e^{-V(\lambda)} \; d\lambda) = 4/\alpha^2$. 

Of course we may consider, when $M \geq n$,  $$\eta_n(d\lambda_1,...,d\lambda_n)= \left(\int \, e^{-V(\lambda)} \, d\lambda_{n+1} ... d\lambda_M\right) \, d\lambda_1 .. d\lambda_n$$ which is a new log-concave distribution according to Prekopa-Leindler theorem. It is still isotropic  up to a dilation of scale $\alpha$ and $$\int \, \exp\left(-(1/\beta)\sum_{i=1}^n \, |X_i|^2 \, \lambda_i^2\right) \;  e^{-V(\lambda)} \; d\lambda = \int \, \exp\left(-(1/\beta)\sum_{i=1}^n \, |X_i|^2 \, \lambda_i^2\right) \;  \eta_n(d\lambda_1, ... ,d\lambda_n) \, .$$ In addition the Poincar\'e constant of $\eta_n$ is less than the one of $e^{-V} \, d\lambda$. We may thus argue as in Example \ref{examplegauss} when $\alpha=1$, and then use the dilation property of the Poincar\'e inequality to conclude that for all $\beta$ and $X$, $$C_P(\mu) \, \leq \, C \, \frac{n^{2/3}}{\alpha^2} \, ,$$ for some universal constant $C$. We can now follow what we did previously to get 
\begin{theorem}\label{thmregress2}
There exist two universal positive constants $c$ and $C$ such that, provided $$\frac{n^{1/3}}{\beta \, \alpha} \, \left(\sup_{j=1,...,M} \, |\langle Y,\tilde X^j\rangle |+\beta\tau\right) \leq c$$ then $$C_P(\nu_{n,\beta}) \, \leq \; C \, \frac{\beta}{\sup_{j=1,...,M} \, |\langle Y,\tilde X^j\rangle |+\beta\tau} \, .$$
\end{theorem}
\medskip

\subsection{Parameter identification.\\ \\}\label{subsecidentify}

We shall briefly indicate another problem, parameter identification via a Bayesian approach as studied in the recent \cite{GPP}. Given a convex function $U: \mathbb R^q \times \mathbb R^d \, \mapsto \, \mathbb R$ one considers the family of probability densities $$f_\theta(x):= f(x,\theta) = Z_\theta^{-1} \, e^{-U(x,\theta)}$$ on $\mathbb R^q$. Given the observation of a sample $X=(X_1,...,X_n)$ of i.i.d. random vectors with density $f_{\theta^*}$, one wants to estimate the unknown parameter $\theta^*$, here again using a Bayesian procedure. 

For a prior distribution density $\pi_0(\theta)$, the posterior density is thus
\begin{equation}\label{eqestim1}
\pi_n(\theta) \, = \, \pi_0(\theta) \; \prod_{i=1}^n \, f_\theta(X_i) \, ,
\end{equation}
and the natural bayesian estimator of $\theta^*$ is once again given by
\begin{equation}\label{eqestim2}
\hat \theta_n = \int \theta \, \pi_n(\theta) \, d\theta \, .
\end{equation}
It is shown in \cite{GPP} that under mild assumptions this estimator is consistent and a bound for the $\mathbb L^p$ error is given, provided there exists a constant $C_P^U$ such that
\begin{equation}\label{eqestim3}
C_P(f_\theta(x) \, dx) \leq C_P^U \quad \textrm{ for all } \theta\in support(\pi_0) \, ,
\end{equation}
An important case where this assumption is satisfied is the location problem, i.e. when $U(x,\theta)=V(x-\theta)$ for some convex function $V$, and of course $q=d$.

Here again the authors propose to use a Langevin Monte Carlo procedure to compute $\hat \theta_n$, and as in the previous subsection the problem is now to estimate the Poincar\'e constant of the measure $\pi_n(\theta) \, d\theta$. The situation is of course much simpler here if one chooses $\pi_0$ as a strictly log-concave log concave distribution since we obtain a bound that only depends on the curvature of $\pi_0$. The case of a general log-concave measure $\pi_0$ is studied in \cite{GPP}.
\bigskip

{\bf Acknowledgements}\\
This work has been (partially) supported by the Project EFI ANR-17-CE40-0030 of the French National Research Agency.

\bibliographystyle{plain}
\bibliography{CG-Regress}
\end{document}